\documentclass[a4paper]{article}
\usepackage[utf8]{inputenc}

\usepackage{hyperref}

\usepackage{tikz}
\usepackage{amsthm}
\usepackage{amsopn}
\usepackage{amsmath}
\usepackage{amssymb}

\usepackage{enumerate}

\usepackage{color}

\renewcommand\P{\ensuremath{\mathcal{P}}}
\newcommand\G{\ensuremath{\mathcal{G}}}
\newcommand\R{\ensuremath{\mathcal{R}}}

\newcommand{\class}[1]{\text{\sc{#1}}}
\newcommand{\free}[1]{\ensuremath{#1}\class{-free}}

\renewcommand\Pr{\ensuremath{\mathbf{P}}}

\theoremstyle{plain}
\newtheorem{thm}{Theorem}

\newtheorem{lemma}[thm]{Lemma}
\newtheorem{prop}[thm]{Proposition}

\theoremstyle{definition}
\newtheorem{definition}{Definition}

\newtheorem{claim}{Claim}
\newcommand\resetClaimCounter{\setcounter{claim}{0}}

\newcommand{\sst}[2]{\left\{\;#1 \;\middle|\; #2 \;\right\}}

\newcommand{\T}{\mathcal{T}}

\title{Chordal graphs are easily testable}
\author{Rémi de Joannis de Verclos
  \thanks{Department of Mathematics, Radboud University Nijmegen, Netherlands. 
    Email: \protect\href{mailto:r.deverclos@math.ru.nl}{\protect\nolinkurl{r.deverclos@math.ru.nl}}. Supported by a Vidi grant (639.032.614) of the Netherlands Organisation for Scientific Research (NWO).}
}
\date{\today}

\begin{document}

\maketitle

\begin{abstract}
  We prove that the class of chordal graphs is easily testable in the following sense.
  There exists a constant~$c>0$ such that,
  if adding/removing at most~$\epsilon n^2$ edges to a graph~$G$
  with~$n$ vertices
  does not make it chordal, then a set of~$(1/\epsilon)^c$ vertices of~$G$
  chosen uniformly at random induces a graph that is not chordal
  with probability at least~$1/2$.
  This answers a question of Gishboliner and Shapira.
\end{abstract}

%\tableofcontents

\section*{Introduction}

A graph~$G$ on~$n$ vertices is \emph{$\epsilon$-far}
from satisfying a property~$\P$ if one has to
add or delete at least $\epsilon n^2$ edges to~$G$
to obtain a graph satisfying~$\P$.
A hereditary class~$\P$ of graphs is \emph{testable}
if for every fixed~$\epsilon>0$ there is a size~$m_\epsilon$
such that the following holds.
If~$G$ is $\epsilon$-far from~$\P$ then
a set~$X\subseteq V(G)$ sampled uniformly at random
among all subsets of~$V(G)$
of size~$m_\epsilon$ induces a graph $G[X]$ that is not in~$\P$
with probability at least~$\frac{1}{2}$.
The property~$\P$ is \emph{easily testable} if moreover~$m_\epsilon$ is
a polynomial function of~$\epsilon^{-1}$.
Otherwise, $\P$ is \emph{hard to test}.

A fundamental result on one-sided testability of graph properties
of Alon and Shapira states that
every hereditary property has a one-sided tester~\cite{AlonS08}.
The proof of this result uses a strengthening
of Szemer\'edi's regularity lemma and gives a query complexity
that is a tower of towers of exponentials of size polynomial in $1/\epsilon$.
This bound was improved by Conlon and Fox~\cite{ConlonF11}
to a single tower of exponentials.

Alon and Shapira also showed that
the class~\free{H} of graphs without induced copy of~$H$
is hard to test when~$H$
is different from $P_2$, $P_3$, $P_4$, $C_4$
and different from the complement of one of these graphs~\cite{AlonS06}.
Moreover, this class is known to be easily testable when~$H\in\{P_2, P_3, P_4\}$,
and thus when~$H$ is the complement of one of these graphs
\cite{AlonS06,AlonF15}.

As a consequence, the graphs~$H$ for which the class~\free{H}
is easily testable are
known, except when $H$ is~$C_4$ or its complement~$\bar{C_4}=2K_2$.
This last question, whether \free{C_4} is easily testable, remains open.
Recently, Gishboliner and Shapira~\cite{GishS18} gave progress
on that question
by showing that every graph that is $\epsilon$-far from being $C_4$-free
contains at least~$n^4/2^{(1/\epsilon)^c}$ induced
copies of~$C_4$ for some constant~$c$, which implies that
$\free{C_4}$ can be tested with query complexity~$2^{(1/\epsilon)^c}$.

The class of chordal graphs is an important and natural subclass of~$\free{C_4}$.
A graph is \emph{chordal} if it contains no induced~$C_k$ for every~$k\geq4$.
Gishboliner and Shapira proved that the class of chordal graph
is testable with query complexity~$2^{(1/\epsilon)^c}$ and
they conjectured that this bound can be further improved
to a polynomial in~$1/\epsilon$~\cite{GishS18}.
In this paper, we confirm this conjecture.
\begin{thm}\label{thm:main}
  The class of chordal graph is testable with query complexity~$O(\epsilon^{-37})$.
\end{thm}
In particular, the class of chordal graph is easily testable.

\subsubsection*{Structure of the paper}

Theorem~\ref{thm:main} is proved in Section~\ref{sec:main proof}.
The main ingredient of the proof is a generalization of the testability of
the $k$-coloring problem (Theorem~\ref{thm:coloring})
which is described in Section~\ref{sec:coloring}.
This result is later used to deal with the "global structure" of chordal graphs.
In Section~\ref{sec:M2-free}, we show various simple properties
about what we call $M_2$-free graphs that are useful to deal
with the local structure of chordal graphs.
In Section~\ref{sec:nearly simplicial},
we show a technical lemma (Lemma~\ref{lem:nearly simplicial}) to deal with
vertices whose neighborhood is nearly a clique.
In Section~\ref{sec:chordal representations}, we show the properties we need
regarding the set of representations of a chordal graph
as an intersection graph of subtrees of a a tree.
The main step toward Theorem~\ref{thm:main} is Lemma~\ref{lem:test pinned},
that shows that it is easy to test if a graph is the intersection graph of a family
of subtrees of a fixed tree with some extra constraints.
The proof of Lemma~\ref{lem:test pinned} is quite involved. It relies on
Theorem~\ref{thm:coloring} as well as lemmas from
Sections~\ref{sec:nearly simplicial} and~\ref{sec:chordal representations}.
The proof of Theorem~\ref{thm:main} in Section~\ref{sec:main proof} is then
essentially an
application of Lemma~\ref{lem:test pinned} and the union bound.

\section{A generalization
  of the coloring problem}
\label{sec:coloring}

The class of $k$-colorable graphs has been proved to be testable
with query complexity~$\frac{k^2\ln k}{\epsilon^2}$
by Goldreich, Goldwasser and Ron~\cite{goldreichgr98}.
This bound was later improved to~$36k\ln k\epsilon^{-2}$
by Alon and Krivelevich~\cite{AlonK02}.
The argument is actually very generic and applies to other graph classes.
Nakar and Ron recently extended this result to a wider family of graph partition
problems that for instance includes split graphs~\cite{NakarR18}.
We give a further generalisation of this result in Theorem~\ref{thm:coloring}.
This theorem is one of the main ingredients of the proof of Theorem~\ref{thm:main}.
It may be of independent interest.

We start by describing the type of problems we study.
In short, colors are subsets of~$[k]=\{1,\dots,k\}$, each vertex has its private list
of possible colors,
and conflicts between colors are expressed by some set inclusion conditions.
\begin{definition}
  Given a set of vertices~$V$ of size~$n$ and a natural number~$k$,
  a \emph{set coloring problem} is given by
  \begin{enumerate}
  \item for every~$v\in V$, a non-empty
    list of \emph{colors}~$L_v \subseteq 2^{[k]}$;
  \item and for every~$(u,v)\in V^2$ with~$u\neq v$,
    two functions~$m_{uv}:L_u \to 2^{[k]}$
    and~$M_{uv}:L_u \to 2^{[k]}$ with
    $m_{uv}(c)\subseteq M_{uv}(c)$ for every~$c\in L_u$.
  \end{enumerate}
\end{definition}
A~\emph{coloring} of~$V$ is a function~$\phi:V\to2^{[k]}$
that assigns to each~$v\in V$ a color~$\phi(v)\in L_v$.
This coloring is~\emph{proper} if for every~$(u,v)\in V^2$ with~$u\neq v$,
\begin{equation}\label{eq:valid pair}
  m_{uv}(\phi(u))\subseteq \phi(v)\subseteq M_{uv}(\phi(u)).
\end{equation}
On the contrary,
a pair~$uv \in\binom{V}{2}$
is \emph{conflicting} if~\eqref{eq:valid pair} is not satisfied
for one of~$(u,v)$ and~$(v,u)$.

We can now state the main result of this section.
\begin{thm}[Testability of set coloring problems]\label{thm:coloring}
  For every~$\epsilon>0$,
  if every coloring of~$V$ has
  at least $\epsilon n^2$ conflicting pairs,
  then sampling $X\subseteq V$ of
  size
  \[
    m_{\ref{thm:coloring}}(\epsilon,k)=
    36k\ln\left(\max_{u\in V}|L_u|\right)\epsilon^{-2}\leq 36k^2\epsilon^{-2}
  \]
  induces a problem on~$X$
  without proper coloring with probability $1/2$.
\end{thm}

Before proving Theorem~\ref{thm:coloring}, let us explain
how $k$-colorability of a graph~$G=(V,E)$ is a particular instance of
the set coloring problem on~$V$.
For every~$v\in V$,
let the list of colors~$L_v$
be the set of singletons~$\{i\}$ with~$i\in[k]$.
For the constraints,
define $m_{uv}(c)=\varnothing$ and~$M_{uv}(c)=[k]\setminus c$
in the case where~$uv$ is an edge,
so that~\eqref{eq:valid pair} is satisfied if and only if
$\phi(u) \neq \phi(v)$.
If~$uv\notin E$,
define $m_{uv}(c)=\varnothing$
and $M_{uv}(c)=[k]$, so that~\eqref{eq:valid pair}
is automatically satisfied for this pair.
It is easy to check that a valid coloring of~$V$ for these constraints
is exactly a function~$\phi:V\to\{\{1\},\dots,\{k\}\}$ satisfying
$\phi(u)\neq\phi(v)$ whenever~$uv\in E$.
Since~$\max_{v\in V}|L_v|=k$,
Theorem~\ref{thm:coloring} then implies the bound of Alon and Krivelevich
for the class of $k$-colorable graphs,
i.e. this property is testable with query complexity~$O(k\ln k\epsilon^{-2})$.
Set coloring problems also generalize other similar graph classes,
such as the class of split graphs.

The proof of Theorem~\ref{thm:coloring} is a direct adaptation
of the proof for the testability of $k$-colorability
of Alon and Krivelevich~\cite[Theorem~3]{AlonK02}.
This proof is postponed to the appendix.

%===============================================================

\section{Nearly simplicial vertices }
\label{sec:nearly simplicial}

A vertex of a graph is \emph{simplicial} if its neighborhood is a clique.
It is well known that the class of chordal graphs is stable by addition of simplicial
vertices (chordal graphs are exactly the graphs that can be obtained from
the empty graph by iteratively adding simplicial vertices~\cite{rose70}).
In this section, we give a relaxed version of this property
(Lemma~\ref{lem:nearly simplicial}).

For a vertex~$v$ of a graph~$G$,
let~$p_G(v)$ be the number of non-edges in the neighborhood of~$v$.
Note that~$v$ is simplicial if and only if~$p_G(v)=0$.
Roughly speaking,
we use the value~$p_G(v)$ as a measure of how close to being simplicial~$v$ is.

The purpose of this section is to prove the following property.
\begin{lemma}\label{lem:nearly simplicial}
  Let~$\epsilon > 0$ and~$n \geq\epsilon^{-1}$.
  Let~$G$ be a graph with~$n$ vertices and with a vertex partition~$X\cup Y$
  such that~$G[Y]$ is chordal and~$p_G(v)\leq \epsilon n^2$ for every~$v\in X$.
  Then~$G$ is $6\epsilon^{1/2}$-close from a chordal graph.
\end{lemma}

The proof uses the following result,
that shows that a dense chordal graph has a large clique.
\begin{lemma}[Gy\'arf\'as, Hubenko, Solymosi~\cite{GyarfasHS02}]
  \label{lem:dense set in chordal}
  Let~$G$ be chordal graph with~$n$ vertices and
  at least~$c\cdot n^2$ edges, then
  \[
    (1-\sqrt{1-2c})n \leq \omega(G).
  \]
\end{lemma}
Given a graph~$G$, a set of vertices~$A\subseteq V(G)$ and a vertex~$u$ of~$G$,
we write~$N_A^G(u)$ the neighborhood of~$u$ in~$A$ for the graph~$G$,
that is the set of vertices~$v\in A$ with~$uv\in E(G)$.
We write~$N_A(u)$  when there is no ambiguity on the graph involved.

We can now proceed to the proof of Lemma~\ref{lem:nearly simplicial}.
\begin{proof}[Proof of Lemma~\ref{lem:nearly simplicial}]
  \resetClaimCounter
  For a vertex~$u \in V$ and a set~$A\subseteq V$,
  let $q_A(u)$ be the number of $P_3$
  on vertices~$\{u, a, v\}$ with middle vertex~$a$
  belonging to~$A$ and where~$v$ is any vertex of~$V$.
  
  \begin{claim}\label{claim:minimum maximal q}
    For every non-empty~$A \subseteq X$ there is $u\in A$
    such that~$q_A(u) \leq 2\epsilon n^2$.
  \end{claim}
  \begin{proof}
    For every~$a\in A$, note that~$p_G(a)$
    is the number of induced~$P_3$ of~$G$ on vertices~$\{u,a,v\}$
    with middle vertex~$a$ and with~$u,v\in V$. 
    By double counting,
    \[
      \sum_{u \in V}q_A(u) = 2\sum_{a \in A}p_G(a)\leq |A|\cdot 2\epsilon n^2,
    \]
    and further
    \[
      \frac{1}{|A|}\cdot\sum_{u \in A}q_A(u) \leq 2\epsilon n^2.
    \]
  \end{proof}
  \begin{claim}\label{claim:clique partition}
    There is a partition~$\bigcup_{i=1}^kX_i$ of~$X$ and
    vertices~$(x_i)_{i=1}^k$ with~$x_i\in X_i$ such that
    $G$ is $4\epsilon^{1/2}n^2$-close to the graph $H$
    defined by the following properties.
    \begin{itemize}
    \item $H[X]$ is a disjoint union of the cliques~$X_1,\dots,X_k$;
    \item for every~$u\in X_i$, $N_Y^{H}(u)=N_Y^G(x_i)$; and
    \item $G[Y]$ and~$H[Y]$ are identical.
    \end{itemize}
  \end{claim}
  \begin{proof}
    We construct the partition iteratively.
    For some~$i$, assume that we have constructed~$X_1,\dots,X_{i-1}$
    and~$x_i,\dots,x_{i-1}$,
    and let us construct~$X_{i}$.
    Define~$A_i=X\setminus(\bigcup_{j=1}^{i-1}X_j)$ and
    assume that~$A_i\neq\varnothing$.
    We distinguish two cases.
    \begin{enumerate}
    \item\label{e:singleton}
      If there is a vertex~$u \in A_i$ with~$d_{A_i}(u)\leq \epsilon^{1/2}n$,
      set~$x_i=u$ and
      define~$X_{i}$ as the singleton~$\{x_i\}$.
    \item\label{e:clique}
      Otherwise, choose~$x_i\in A_i$ such that~$q_{A_i}(x_i)\leq 2\epsilon n^2$
      and set~$X_{i}:= \{x_i\} \cup N_{A_i}(x_i)$.
      The existence of such a vertex~$x_i$ is ensured
      by Claim~\ref{claim:minimum maximal q}.
    \end{enumerate}
    Assume now that~$X_1,\dots,X_k$ and~$x_1,\dots,x_k$ are defined.
    Let~$H$ be the graph described in the claim and
    let us estimate the number of edges in~$E(G)\triangle E(H)$.

    For every~$i\in[k]$,
    let~$m_i$ be the number of
    edges of~$G$ between~$X_{i}$
    and non-neighbors of~$x_i$ in~$A_i\cup Y$
    plus the number of missing edges inside~$N^G_{A_i\cup Y}(x_i)$.
    Note that~$|E(G)\triangle E(H)|=\sum_{i=1}^km_i$.
    We estimate~$m_i$ depending on the case
    chosen in the construction of~$X_i$.
    
    In Case~\ref{e:singleton},
    $m_i$ counts only the at most~$d_{A_i}(u) \leq \epsilon^{1/2}n$
    edges from~$x_i$ to~$A_i$.
    Since this happens at most~$n$ times in the process,
    the sum of~$m_i$ over every index~$i$ corresponding to
    Case~\ref{e:singleton} is at most~$\epsilon^{1/2}n^2$.
    
    In Case~\ref{e:clique},
    we claim that~$m_i\leq q_{A_i}(x_i) + p_G(x_i) \leq 3\epsilon n^2$.
    Indeed, every missing edge of~$G[N_{A_i\cup Y}(x_i)]$
    contributes for one in~$p_G(x_i)$;
    and every edge between~$X_{i}$ and a non-neighbor of~$x_{i}$
    form a~$P_3$ that contributes for one in~$q_{A_i}(x_i)$.
    Moreover, in Case~\ref{e:clique} it holds that
    $|X_{i}| = d_{A_i}(u) + 1 > \epsilon^{1/2}n$,
    so Case~\ref{e:clique} occurs at most~$\epsilon^{-1/2}$ times.
    As a consequence, the total contribution of these cases is
    at most
    $\epsilon^{-1/2} \cdot 3\epsilon n^2 = 3\epsilon^{1/2}n^2$.
    
    To sum up,
    the total number of edges in~$E(G)\triangle E(H)$ is at most
    $\sum_{i=1}^km_i\leq \epsilon^{1/2}n^2 + 3\epsilon^{1/2}n^2 = 4\epsilon^{1/2}n^2$,
    which proves the claim.
  \end{proof}

  \begin{claim}\label{claim:clique size}
    For every~$u\in X$,
    $N_Y(u)$ contains a clique of size at least
    $d_Y(u) - 2\epsilon^{1/2}n$.
  \end{claim}
  \begin{proof}
    Let~$d=d_Y(u)$ be the degree of~$u$ in~$Y$.
    The claim holds with the empty clique if $d \leq (2\epsilon)^{1/2}n$,
   so we assume that~$d > (2\epsilon)^{1/2}n$.
   We aim to apply Lemma~\ref{lem:dense set in chordal} on the graph
   $F=G[N_Y(u)]$ of size~$d$.
   The number of edges in~$F$ is at least
   \[
     \binom{d}{2} - p_G(u) \geq \frac{d^2}{2} -\frac{d}{2} - \epsilon n^2
     = d^2\left(\frac{1}{2}-\frac{\epsilon n^2}{d^2}-\frac{1}{2d}\right)
     \geq d^2\left(\frac{1}{2} -2\frac{\epsilon n^2}{d^2}\right),
   \]
   where the last estimation comes from $\epsilon n \geq 1$ and~$d \leq n$.
   Lemma~\ref{lem:dense set in chordal} applied to~$F$ with
   $c=\frac{1}{2} - 2\epsilon\frac{n^2}{d^2}$
    then gives~$\omega(F)\geq d - 2\epsilon^{1/2}n$, which proves the claim.
  \end{proof}
  We are now ready to finish the proof.
  For every~$i\in[k]$,
  Claim~\ref{claim:clique size} provides
  a clique~$C_i\subseteq N_Y^G(x_i)$ of~$G$
  with $d_Y^G(x_i) - |C_i| \leq (2\epsilon)^{1/2}n< 2\epsilon^{1/2}n$.
  Let~$H'$ be the graph obtained from~$H$ by deleting
  for each~$i$ every edge between~$X_i$
  and~$N_Y^G(x_i) \setminus C_i$, so that~$N_Y^{H'}(x_i)$
  is the clique~$C_i$.
  Since, $N_Y^{H}(u)=N_Y^{G}(x_i)$ and~$|N_Y^{G}(x_i)\setminus C_i|\leq2\epsilon^{1/2}n$
  for every~$u\in X_i$,
  the total number of edges in~$E(H)\triangle E(H')$ is at most~$2\epsilon^{1/2}n^2$.
  As a consequence,
  \[
    |E(G)\triangle E(H')|\leq |E(G)\triangle E(H)| + |E(H)\triangle E(H')|
    \leq 6\epsilon^{1/2}n^2.
  \]
  It remains to show that~$H'$ is a chordal graph.
  To see this, note that~$H'[Y]=G[Y]$
  so~$H'[Y]$ is chordal.
  Moreover,
  it follows from the construction that every vertex of~$X$ is simplicial in~$H'$
  since~$X$ is a disjoint union of cliques.
  This implies that~$H'$ is chordal and concludes the proof of the lemma.
\end{proof}

Note that the proof of Lemma~\ref{lem:nearly simplicial}
relies only on two properties of chordal graphs:
the existence of a big clique in dense sets
--given by Lemma~\ref{lem:dense set in chordal}-- and
the stability of chordal graphs by addition of simplicial vertices.
As an equivalent of Lemma~\ref{lem:dense set in chordal} holds for $C_4$-free
graphs (see~\cite{GyarfasHS02}) and adding a simplicial vertex does not
create an induced~$C_4$,
an equivalent to Lemma~\ref{lem:nearly simplicial}
could also be derived for $C_4$-free
graphs.

\section{\texorpdfstring{$M_2$}{M2}-free graphs}
\label{sec:M2-free}

For a graph~$G=(V,E)$ and two disjoint sets~$L,R\subseteq V$,
we write~$G[L,R]$ the bipartite graph with parts~$L$ and~$R$
and edge set~$\sst{\ell r\in E}{\ell\in L\text{ and }r\in R}$.
In such a bipartite graph,
we call~$M_2$ an induced bipartite matching of size~$2$,
that is a set~$\{\ell_1,\ell_2,r_1,r_2\}$ of distinct vertices
with $\ell_i\in L$, $r_i\in R$
and~$\ell_ir_i\in E(G)$ for~$i\in\{1,2\}$ and~$\ell_1r_2,\ell_2r_1\notin E(G)$,
as in Figure~\ref{fig:M2}.

\begin{figure}[b]
  \centering
  \begin{tikzpicture}[scale=0.4]
    \tikzstyle{node}=[circle,draw,fill=black,scale=0.5];
    \tikzstyle{edge}=[thick];
    \foreach\y in {-1,1} {
      \node[node] (A) at (-1, \y) {};
      \node[node] (B) at (1, \y) {};
      \draw[edge] (A)--(B);
    }
    \draw[dashed, gray] (-1.05, 0) ellipse (0.7 and 2.4);
    \draw[dashed, gray] (1.05, 0) ellipse (0.7 and 2.4);
    \node at (-1,-3.2) {$L$};
    \node at (1,-3.2) {$R$};
  \end{tikzpicture}
  \caption{The induced subgraph $M_2$.}\label{fig:M2}
\end{figure}
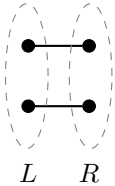

In this section, we describe the structure of
$M_2$-free bipartite graphs (Theorem~\ref{thm:structure of M2-free graphs})
and we show that they are testable with query
complexity~$O(\frac{1}{\epsilon}\ln\frac{1}{\epsilon})$
(Theorem~\ref{thm:testability of M2-free graphs}).

Let~$G$ be a graph and $L$ and~$R$ be two sets of vertices.
A vertex~$v$ of~$L$ is \emph{peelable in $G[L,R]$} if $N_R(v)=\varnothing$;
a vertex~$v$ of~$R$ is \emph{peelable in $G[L,R]$} if~$N_L(v)=L$.
For an integer~$k$, a vertex~$v\in L\cup R$ is \emph{$k$-peelable in~$G[L,R]$}
if either ($v\in L$ and~$|N_R(v)|\leq k$)
or ($v\in R$ and~$|L\setminus N_L(v)|\leq k$).
\begin{thm}[Structure of $M_2$-free graphs]\label{thm:structure of M2-free graphs}~
  Let~$G$ be a graph on~$L\cup R$.
  The following four properties are equivalent.
  \begin{enumerate}[(i)]
  \item\label{it:no M2} $G[L,R]$ contains no $M_2$.
  \item\label{it:peelable}
    For every subsets~$L_0\subseteq L$ and~$R_0\subseteq R$
    such that~$L_0\cup R_0$ is non-empty,
    there is a vertex of~$L_0 \cup R_0$
    that is peelable in~$G[L_0,R_0]$.
  \item\label{it:elimination sequence}
    There is an enumeration~$v_1,\dots,v_p$
    of~$R\cup L$ such that for every~$v_i\in L$ and~$v_j\in R$,
    $v_iv_j$ is an edge of~$G$ if and only if~$j<i$.
  \item\label{it:intervals}
    There is a family of intervals~$(I_v)_{v\in L\cup R}$ of~$[0,1]$
    such that~$0\in I_u$ for every~$u\in L$, $1\in I_v$ for every~$v\in R$,
    and~$I_u$ and~$I_v$ intersect if and only if~$uv$ is an edge of~$G$.
  \end{enumerate}
\end{thm}
\begin{proof}
  $\eqref{it:no M2}\Rightarrow\eqref{it:peelable}$.
  Assume for a contradiction that~$G[L_0,R_0]$ has no peelable vertex,
  that is every vertex of~$L_0$ has a neighbor in~$R_0$
  and every vertex of~$R_0$ has a non-neighbor in~$L_0$.
  Since at least one of~$L_0$ and~$R_0$ is non-empty,
  the assumption above implies that both of~$L_0$ and~$R_0$ are non-empty.
  Let us show that~$G[L_0,R_0]$ contains an induced $M_2$.
  Let~$\ell_1$ be a vertex of~$L_0$ that minimizes~$|N_{R_0}(\ell_1)|$ and
  take~$r_1\in N_{R_0}(\ell_1)$.
  By assumption, $r_1$ has a non-neighbor~$\ell_2$ in~$L_0$.
  By the construction of~$\ell_1$, it holds that
  $|N_{R_0}(\ell_1)|\leq|N_{R_0}(\ell_2)|$. Since~$r_1$ belongs to
  $N_{R_0}(\ell_1)\setminus N_{R_0}(\ell_2)$, there exists~$r_2$
  in~$N_{R_0}(\ell_2)\setminus N_{R_0}(\ell_1)$.
  The quadruple~$\{\ell_1,\ell_2,r_1,r_2\}$ then forms an induced~$M_2$
  of~$G[L_0,R_0]$.
  
  $\eqref{it:peelable}\Rightarrow\eqref{it:elimination sequence}$.
  To construct the sequence~$v_1,\dots,v_p$, we start with the set~$V_1=R\cup L$
  and we let~$v_{i}\in V_i$ be a peelable vertex of~$G[R\cap V_i, L\cap V_i]$
  as long as~$V_i$ is non-empty.
  Such a vertex exists because of~$\eqref{it:peelable}$.
  We then define~$V_{i+1}:=V_i\setminus\{v_i\}$.
  Assume now that~$v_i\in L$ and~$v_j\in R$ for some indices~$i$ and~$j$
  and let us show that $v_iv_j$ is an edge if and only if~$j<i$.
  It follows from the constuction and the definition of peelable that
  $N_{R \cap V_i}(v_i)=\varnothing$ and~$N_{L\cap V_j}(v_j)=L\cap V_j$.
  If~$i<j$, then $v_j\in R\cap V_i$, and further~$v_j\notin N(v_i)$
  by the first equality above.
  Similarly, if~$i>j$, then $v_i\in L\cap V_j$, and further~$v_i\in N(v_j)$.

  $\eqref{it:elimination sequence}\Rightarrow\eqref{it:intervals}$.
  It suffices to define~$(I_{v})_{v\in L\cup R}$ as follows.
  For every~$i\in[p]$, set~$I_{v_i}=[0,\frac{i}{p}]$ if~$v_i\in L$
  and~$I_{v_i}=[\frac{i}{p},1]$ if~$v_i\in R$.
  To deduce~$\eqref{it:intervals}$ from~$\eqref{it:elimination sequence}$,
  it then suffices to notice that~$[0,\frac{i}{p}]$ and~$[\frac{j}{p}, 1]$
  intersect if and only if~$j < i$.
  
  $\eqref{it:intervals}\Rightarrow\eqref{it:no M2}$.
  Consider the interval graph~$G'$ on~$L\cup R$ which
  is the intersection graph of the family~$(I_v)_{v\in L\cup R}$.
  Note that~$L$ and~$R$ are clique of~$G'$ and that
  by~$\eqref{it:intervals}$
  the bipartite graphs~$G[L,R]$
  and~$G'[L,R]$ are identical.
  It follows that if~$G$ contains an induced~$M_2$
  with edges~$\ell_1r_1$ and~$\ell_2r_2$,
  then $\ell_1r_1r_2\ell_2$ is an induced cycle of~$G'$, which is impossible as~$G'$ is an interval graph.
\end{proof}

Because of its decomposition structure,
$M_2$-free graphs are testable with a good
query complexity.
This is proved in Theorem~\ref{thm:testability of M2-free graphs}.
If one is not interested in having an explicit exponent in
the query complexity, this theorem can also be deduced from the
regularity lemma of Alon, Fischer and Newman~\cite{AlonFN07}.
\begin{thm}\label{thm:testability of M2-free graphs}
  Let~$G$ be a bipartite graph on~$V=L\cup R$ with~$|V|\leq n$.
  If one has to change at least~$\epsilon n^2$ edges to~$G[L,R]$
  to make it~$M_2$-free,
  then a set~$X\subseteq V$ chosen uniformly at random
  among subsets of~$V$
  of size~$m_{\ref{thm:testability of M2-free graphs}}(\epsilon)=\frac{4}{\epsilon}\ln\frac{1}{\epsilon}$
  induces a bipartite graph~$G[L\cap X, R\cap X]$ that contains a~$M_2$
  with probability at least~$\frac{1}{2}$.
\end{thm}

\begin{proof}
  First note that there are at most~$n^2/4$ edges between~$L$ and~$R$,
  so the hypothesis does not hold if~$\epsilon\geq1/4$. We may therefore assume
  that~$\epsilon<1/4$.
  
  We iteratively peel the vertices of~$G$ as follows:
  we start with~$i=1$ and the set of vertices~$V_1=L\cup R$.
  As long as the the bipartite graph~$G[L\cap V_i,R\cap V_i]$
  has a $\epsilon n$-peelable vertex~$v_i$,
  we set~$V_{i+1}=V_{i}\setminus\{v_i\}$.
  We then reiterate with~$i:=i+1$ until $V_i$ contains no $\epsilon n$-peelable
  vertex.
  This gives a list~$v_1,\dots,v_\ell$ of vertices
  such that~$v_i$ is $\epsilon n$-peelable in~$G[L\cap V_i,R\cap V_i]$
  and~$V_i=(R\cup L)\setminus\sst{v_j}{1\leq j<i}$ for every~$i\in[\ell]$.
  It also follows from the construction that the final set~$V_\ell$ contains no peelable vertex.
  
  If~$V_\ell$ is empty,
  we construct a $M_2$-free bipartite graph~$H$ on~$L\cup R$ from~$G[L,R]$
  as follows:
  for every~$i\in[\ell]$, we add every missing edge from~$v_i$ to~$V_i\cap L$ if~$v_i\in R$
  or we delete every edge from~$v_i$ to~$V_i\cap R$ if~$v_i\in L$ for every~$i\in[\ell]$. 
  Since each~$v_i$ is $\epsilon n$-peelable in~$G[L\cap V_i, R\cap V_i]$,
  the operation above adds/deletes at most~$\epsilon n$ edges for each vertex,
  so at most~$\epsilon n^2$ in total.
  This ensures that~$v_i$ is peelable in~$H[L\cap V_i,R\cap V_i]$ for every~$i\in [\ell]$,
  so by Theorem~\ref{thm:structure of M2-free graphs}\eqref{it:elimination sequence},
  the bipartite graph~$H$ is $M_2$-free, which contradicts the hypothesis.
  
  We now assume that~$V_\ell$ is non-empty
  and set~$L_0=L\cap V_\ell$ and~$R_0=R\cap V_\ell$.  
  Since the vertices of~$V_\ell$ are not $\epsilon n$-peelable ,
  it must hold that~$|V_\ell|>\epsilon n$.
  By Theorem~\ref{thm:structure of M2-free graphs}\eqref{it:peelable},
  $G[L_0\cap X, R_0\cap X]$ can be $M_2$-free only if it has a peelable vertex.
  The probability that one fixed vertex~$v\in X\cap V_\ell$
  is peelable in~$G[L_0\cap X, R_0\cap X]$
  is the probability that~$X\setminus\{v\}$ does not intersect
  the set~$N_{R_0}(v)$ if~$v\in L$ or
  the set~$L_0\setminus N_{L_0}(v)$ if~$v\in R$.
  In both case, this set has size at least~$\epsilon n$.
  Therefore, the probability that~$v$ is peelable in~$G[L_0\cap X, R_0\cap X]$
  is at most~$(1-\epsilon)^{|X|-1}$.
  Moreover, the probability that~$X\cap V_\ell\neq\varnothing$
  is at most~$(1-\epsilon)^{|X|}$.

  By the union bound, the probability that~$G[L_0\cap X, R_0\cap X]$
  is $M_2$-free is at most
  \begin{align*}
    |X|\left(1-\epsilon\right)^{|X|-1}
    +\left(1-\epsilon\right)^{|X|}
    & = \left(\frac{|X|}{1-\epsilon}+1\right)\cdot\left(1-\epsilon\right)^{|X|}\\
    & < 3|X|\cdot e^{-\epsilon|X|}.
  \end{align*}
  For the last estimation we used that~$\epsilon < \frac{1}{4}$.
  With $|X|=\frac{4}{\epsilon}\ln\frac{1}{\epsilon}$,
  and using that~$\epsilon<\frac{1}{4}$
  and~$\epsilon\ln\frac{1}{\epsilon}\leq\frac{1}{e}$, this bound becomes
  \[
    \frac{12}{\epsilon}\ln\frac{1}{\epsilon}\cdot\epsilon^{4}
    \leq \frac{12}{e}\epsilon^2
    \leq \frac{12}{16e} < \frac{1}{2}.
  \]
  This concludes the proof.
\end{proof}

\section{Chordal graph representations}
\label{sec:chordal representations}

Chordal graphs are exactly the intersection graphs of subtrees of a
tree~\cite{gavril74}.
For our purpose, it is convenient to express this with topological trees.
A \emph{topological tree}~$\T$ is the topological space
described by a graph~$G$ that is a tree.
The \emph{leaves} of such a topological tree are the points of~$\T$ associated
to the leaves (in the graph sense) of~$G$.
A subtree~$T$ of a topological tree~$\T$ is a non-empty connected
subset of~$\T$.
In this case,~$T$ is also a topological tree.

A graph~$G=(V,E)$ is chordal if and only if there is a topological tree~$\T$
and a family~$(T_v)_{v\in V}$ of subtrees of~$\T$ such that~$uv\in E$
if and only if~$T_u\cap T_v\neq\varnothing$ for every distinct~$u,v\in V$.
In this case, the family~$(T_v)_{v\in V}$ is
a \emph{chordal representation} of~$G$.

Chordal representations can be simplified using the following property.
\begin{prop}\label{prop:nonempty tree intersection}
  Let~$(T_v)_{v\in V}$ be a chordal representation on~$\T$ of a graph~$G=(V,E)$ and
  let~$\T'$ be a subtree tree of~$\T$
  such that~$\T'\cap T_v\neq\varnothing$ every~$v\in V$.
  Then~$(T_v\cap\T')_{v\in V}$ is a chordal representation of~$G$ on~$\T'$.
\end{prop}
\begin{proof}
  The hypothesis ensures that~$T_u':=T_u\cap \T'$ is indeed a subtree of~$\T'$
  for every~$u\in V$.
  Let us show that for every pair of distinct vertices~$u,v\in V$,
  the trees~$T_u'$ and~$T_v'$ intersect if and only if~$uv\in E$.

  Since~$T_u'$ and~$T_v'$ are subsets of~$T_u$ and~$T_v$ respectively,
  it is clear that if~$uv\notin E$, then $T_u\cap T_v=\varnothing$
  so~$T_u'\cap T_v'=\varnothing$
  Now, assume that~$uv\in E$.
  In this case, the trees~$T_u$ and~$T_v$ intersect,
  and we know from the assumption that both of~$T_u$ and~$T_v$
  intersect the tree~$\T'$.
  As trees have the Helly property,
  it follows that~$T_u\cap T_u\cap \T'=T_u'\cap T_v'$ is indeed non-empty.
\end{proof}

A chordal representation~$(T_v)_{v\in V}$ on~$\T$ of a graph~$G$ is \emph{minimal}
if there is no strict subtree~$\T'\subsetneq\T$ such that~$(T_v\cap\T')_{v\in V}$
is a chordal representation of~$G$ on~$\T$.
If~$G$ is chordal, then~$G$ has a minimal chordal representation.
Minimal representations
are characterized by the following property.
\begin{prop}\label{prop:leaf}
  A chordal representation~$(T_v)_{v\in V}$ on~$\T$ is minimal
  if and only if for every leaf~$\ell$ of~$\T$ there is a vertex~$v\in V$
  with~$T_v=\{\ell\}$.
\end{prop}
\begin{proof}
  If every leaf~$\ell$ corresponds to a vertex~$v_\ell\in V$
  with~$T_{v_\ell}=\{\ell\}$, it is clear that ~$(T_v)_{v\in V}$
  is minimal because for every strict subtree~$\T'\subsetneq\T$
  there is a leaf~$\ell$ of~$\T$ that is not in~$\T'$.
  In that case,~$T_{v_\ell}\cap\T'=\varnothing$ is not a subtree of~$\T'$.

  Let us prove the other implication.
  Assume that there is a leaf~$\ell$ of~$\T$ such that~$T_v\neq\{\ell\}$
  for every~$v\in V$ and let us show that~$(T_v)_{v\in V}$
  is not minimal.
  In this case, there is a small open neighborhood~$U$ of~$\ell$ in the
  topological space~$\T$ such that~$T_v\setminus U\neq\varnothing$
  for every~$v\in V$ and~$\T':=\T\setminus U$ is a topological tree.
  It follows that the elements of~$(T_v\cap\T')_{v\in V}$
  are not empty.
  By Proposition~\ref{prop:nonempty tree intersection},
  we deduce that~$(T_v\cap\T')_{v\in V}$ is a chordal representation
  of~$G$.
  This proves that~$(T_v)_{v\in V}$ is not a minimal representation of~$G$,
  which concludes the proof.
\end{proof}
In the proof of Theorem~\ref{thm:main},
we need an upper bound on the number of minimal chordal representations of
a graph.
\begin{lemma}\label{lem:chordal repr}
  A graph~$G$ on~$n$ vertices has at most~$m_{\ref{lem:chordal repr}}(n)=(3n)^{2n^2}$
  minimal chordal representations up to homeomorphism.
\end{lemma}
Lemma~\ref{lem:chordal repr} has the right order of magnitude in the sense that there
are graphs on~$n$ vertices with $n^{\Omega(n^2)}$ non-homeomorphic
minimal chordal representations.
Indeed, the graph that consists of a clique
of size~$n/2$ and $n/2$ isolated vertices can be represented on a star
with~$n/2$ branches, taking a leaf for each isolated vertex and where
every vertex of the clique contains the center.
Then, there is $(n/2)!$ ways to chose the order of the bounds of the trees
of the clique on each of the~$n$ branches,
giving in total $((n/2)!)^{n/2}= n^{n^2/4 + O(n)}$ non-isomorphic
minimal chordal representations. See Figure~\ref{fig:star}.

\begin{figure}[b]
  \centering
  \begin{tabular}{c c c}
    \begin{tikzpicture}
      \tikzstyle{node}=[circle,draw,scale=0.5];
      \tikzstyle{edge}=[thick];
      \begin{scope}
        \foreach\a/\c/\i in {0/blue/3, 120/green/1, 240/purple/2}{
          \node[\c,fill,node,label={\a:$v_\i$}] (A\a) at (\a:0.8) {};
        }
      \end{scope}
      \begin{scope}[xshift=2cm]
        \foreach\a/\c/\i in {0/10/6, 120/50/4, 240/90/5}{
          \node[node, fill, yellow!\c!red,label={\a:$v_\i$}] at (\a:0.8){};
        }
      \end{scope}
      \foreach\a in {0, 120, 240}{
        \foreach\b in {0, 120, ..., \a}{
          \draw[edge] (A\a) -- (A\b);
        }
      }
    \end{tikzpicture} &\hspace{2cm} &
    \begin{tikzpicture}[scale=1.2]
      \def\l{0.6mm}
      \tikzstyle{node}=[circle,draw,fill,scale=0.2];
      \tikzstyle{edge}=[thick];               
      \foreach\a/\c/\i in {-30/10/4, 90/50/5, 210/90/6}{
        \draw[gray!20!white, line width=1.4mm] (0,0) -- (\a:1.13);
        \node[circle,draw,fill,scale=0.3, yellow!\c!red,
        label={[yellow!\c!red]\a:$T_{v_\i}$}]
        at (\a:1.1) {};}
      \foreach\a/\r/\dx in {-30/3/1, 90/6/0.5, 210/9/1}{
        \draw[blue, line width=\l, xshift=\dx mm]
        (0,0) -- (\a:0.\r);}
      \foreach\a/\r in {-30/9, 90/3, 210/6}{
        \draw[purple, line width=\l, xshift=0mm]
        (0,0) -- (\a:0.\r);}
      \foreach\a/\r/\dx in {-30/7/-1, 90/9/-0.5, 210/3/-1}{
        \draw[green, line width=\l, xshift=\dx mm]
        (0,0) -- (\a:0.\r);}
      \node[blue] at (30:0.5) {$T_{v_1}$};
      \node[green] at (150:0.5) {$T_{v_2}$};
      \node[purple] at (270:0.5) {$T_{v_3}$};
    \end{tikzpicture} \\
  \end{tabular}
  \caption{
    Left: The chordal graph $G=K_{n/2}\cup E_{n/2}$ with $n=6$.
    Right: One of the $((n/2)!)^{n/2}$ minimal non-isomorphic representations of~$G$ on the star.
  }\label{fig:star}
\end{figure}
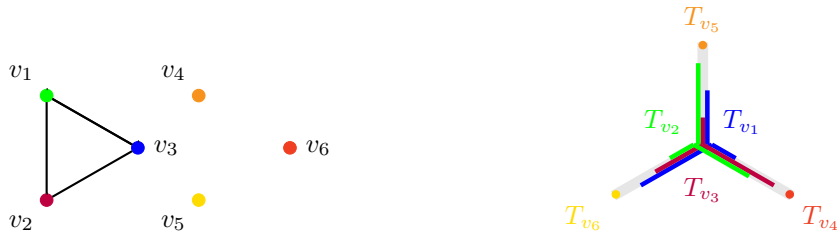

\begin{proof}[Proof of Lemma~\ref{lem:chordal repr}]
  Let $(T_u)_{u\in V}$ be a minimal chordal representation of~$G$ on
  the topological tree~$\T$.
  By Proposition~\ref{prop:leaf}, the tree~$\T$
  has at most~$n$ leaves, and
  thus~$\T$ the union of less than~$2n$ path sections.
  Such a tree can be constructed inductively by adding paths one by one,
  each time choosing among less than~$2n$ possibilities.
  It follows that there are less that~$(2n)^{2n}$ possibilities for~$\T$,
  up to homeomorphism.
  
  Now, a family of~$n$ subtrees~$(T_u)_{u\in V}$ of~$\T$ considered
  up to homeomorphism
  can be constructed inductively by choosing the trees one by one.
  Note that a subtree~$T_u$ also has at most~$n$ leaves
  and that a subtree of~$\T$ is determined by its leaves.
  When constructing~$T_u$, note that there are at most~$2n + n^2\leq 2n^2$
  sections formed by the previously chosen trees.
  As a consequence, there are at most~$(2n^2)^n$ choices
  for the leaves of~$T_u$.
  In total, there are therefore at most~$(2n^2)^{n^2}$ non-equivalent ways
  of choosing~$(T_u)_{u\in V}$.

  In total this gives~$(2n)^{2n}\cdot(2n^2)^{n^2}\leq (3n)^{2n^2}$
  different minimal chordal representations.
\end{proof}

\section{Pinned chordal graph}
\label{sec:pinned}

\begin{definition}
  Let~$\T$ be a topological tree.
  Fix a set~$V$ of vertices and associate to each vertex~$v\in V$
  a point~$x_v\in\T$.
  A graph~$G$ on~$V$ is \emph{$(x_v)_{v\in V}$-pinned}
  if there is a chordal representation~$(T_v)_{v\in V}$ of~$G$
  on the tree~$\T$ such that~$x_v\in T_v$ for every~$v\in V$.
  In this case,~$(T_v)_{v\in V}$ is a
  \emph{$(x_v)_{v\in V}$-pinned representation of~$G$ on~$\T$}.
\end{definition}
Given a topological tree~$\T$
and two elements~$a,b\in\T$, let~$[a,b]$ denote the (unique) shortest path from~$a$
to~$b$ in~$\T$.
Note that~$[a,b]$ is also the shortest path from~$a$ to~$b$ in every tree~$\T'$
such that~$\T$ is a subtree of~$\T'$.
We write~$(a,b)$ for the set~$[a,b]\setminus\{a,b\}$.

Lemma~\ref{lem:YS} shows that the extensibility of
a chordal representation
nearly boils down to the existence of pinned representation.
Lemma~\ref{lem:test pinned} shows that a polynomial tester can distinguish between
chordal graphs and graphs that are~$\epsilon$-far from being pinned.
\begin{lemma}\label{lem:YS}
  Let~$G=(V,E)$ be a graph and $S$ be a set of vertices.
  Let~$(T_v)_{v\in S}$ be a chordal representation on~$\T$ of~$G[S]$.
  Set
  \[
    Y_S=\sst{u\in V}{\text{$N(u)\cap S$
        is not a clique of~$G$}}.
  \]
  There is a set~$\G\subseteq\T$ of size at most~$\binom{|S|}{2}$ and
  a family~$(x_{v})_{v\in Y_S}$ of elements of~$\G$
  with the following property.
  For every~$U\subseteq Y_S$,
  if~$G[U\cup S]$ has a chordal representation that extends~$(T_v)_{v\in S}$
  on a tree~$\T'$ that extends~$\T$
  then~$G[U]$ has a $(x_v)_{v\in U}$-pinned
  representation on~$\T$.
\end{lemma}
\begin{proof}
  Let us first construct~$\G$. For every non-edge~$ab$ of~$G[S]$, the trees~$T_a$ and~$T_b$ do not intersect because~$(T_s)_{s\in S}$ is a chordal representation of~$G[S]$. Let~$y_{ab}$ be the point of~$T_a$ that is the closest to~$T_b$ in the topological space~$\T$. Note that since~$T_{a}$ and~$T_{b}$
  are subtrees of the tree~$\T$,
  every path from a point of~$T_{a}$ to a point of~$T_{b}$ contains~$y_{ab}$.
  Now define $\G = \sst{y_{ab}}{\text{$ab$ is a non-edge of~$G[S]$}}$.
  Since $G[S]$ has at most~$\binom{|S|}{2}$ non-edges, it holds that~$|\G|\leq\binom{|S|}{2}$.
  
  We are now ready to define~$(x_v)_{v\in Y_S}$.
  By definition of~$Y_S$,
  there is a non-edge~$a_vb_v$ of~$G$
  with~$a_v,b_v\in S$  for every~$v\in Y_S$.
  We then define~$x_v=y_{a_vb_v}$.

  It remains to show that~$(x_v)_{v\in Y_S}$ has the required property.
  Take~$U\subseteq Y_S$
  and assume that~$(T_s)_{s\in S}$
  extends to a chordal representation~$(T_s)_{s\in S\cup U}$ of~$G[U\cup S]$ 
  (so we the same trees associated to the elements of~$S$)
  on a tree~$\T'$ that extends~$\T$.

  We first claim that~$x_u\in T_u$ for every~$u\in Y_S\cap U$.
  Indeed, $ua_u$ and~$ub_u$ are edges of~$G$,
  so~$T_u$ intersects both~$T_{a_u}$ and~$T_{b_u}$.
  In particular, $T_u$ contains a path from a point of~$T_{a_u}$
  to a point of~$T_{b_u}$.
  This implies that $T_u$ contains the point~$y_{a_ub_v}=x_u$.

  Define~$T_u'=T_u\cap \T'$ for every~$u\in U$.
  As previously prove, it holds that~$x_u\in\T$, so~$x_u$ belongs to~$T_u\cap\T=T_u'$. This in particular implies that $T_u\cap\T$ is non-empty.
  Consequently, Proposition~\ref{prop:nonempty tree intersection} applies to~$G[U]$ and shows that the family~$(T_u')_{u\in U}$ is a chordal representation of~$G[U]$.
\end{proof}
\begin{lemma}\label{lem:test pinned}
  Let~$\epsilon > 0$ and let~$k$ be a integer.
  Fix a tree~$\T$ with at most~$k$ leaves and a set~$\G_0\subseteq\T$
  of size at most~$k$.
  Then for every graph~$G$ on~$V$
  and values~$(x_v)_{v\in V}$ of~$\G_0$, one of the following holds
  \begin{itemize}
  \item $G$ is $\epsilon$-close to a chordal graph; or
  \item $G[X]$ is not a $(x_v)_{v\in X}$-pinned chordal graph
    with probability at least~$\frac{1}{2}$,
  \end{itemize}
  where $X$ is chosen uniformly at random
  among subsets of~$V$
  of size
  $m_{\ref{lem:test pinned}}(\epsilon,k)=2^{40}\epsilon^{-4}k^6\ln^{4}k$.
\end{lemma}

\begin{proof}
  \resetClaimCounter
  The idea of the proof is the following.
  We first cut the tree~$\T$ into small linear sections.
  Instead of directly testing if there is a $(x_v)_{v\in V}$-pinned
  representation~$(T_v)_{v\in V}$ of~$G$,
  we first test if it is possible to attribute to every vertex~$u$
  the general shape~$\phi(u)$ of a tree~$T_u$,
  essentially defined by the set of sections the tree~$T_u$ touches.
  The problem of finding this shapes and the natural constraints
  associated to them
  (such as the fact that $uv$ is an edges, $T_u$ and~$T_v$
  have to touch a common region)
  can be expressed in terms of a set coloring problem
  as defined in Section~\ref{sec:coloring} and is therefore easy to test.
  Then assuming that such a~$\phi$ exists (with some controlled error),
  we try to construct a family of trees~$T_u$
  by looking the shapes~$\phi(u)$ into a tree by analyzing what happens
  locally.
  This allows us to test the general structure of~$G$.
  Do do so, we need to test the "local" structure of a family of trees~$(T_v)_{v\in V}$.
  When restricting to linear sections of~$\T$,
  a chordal representation of~$(T_v)_{v\in V}$ has to be close to a $M_2$-free graph.
  This property can be tested using Theorem~\ref{thm:testability of M2-free graphs}.
  Assuming that none of these tests fail,
  we then show that it is possible to construct the family of trees~$(T_v)_{v\in V}$
  by gluing the $M_2$-free bipartite graphs obtained on small section to the
  general shapes.

  In the course of the proof, we will need the values
  $y=m_{\ref{thm:testability of M2-free graphs}}(\frac{\epsilon}{6k})$
  and~$\delta=\frac{1}{3 y^2}$.
  The reason of these choices appears in Claim~\ref{claim:close to local interval}.
  The final query complexity obtained is
  $|X|=m_{\ref{lem:test pinned}}(\epsilon,k)=
  \max(m_{\ref{thm:coloring}}(\delta,3k), 4y)=
  m_{\ref{thm:coloring}}(\delta,3k)$.
  This is needed in the proof of Claims~\ref{claim:coloring disproof}
  and~\ref{claim:close to local interval}.
  Estimating~$m_{\ref{lem:test pinned}}(\epsilon,k)$ using the values in
  Lemma~\ref{lem:test pinned} and
  Theorems~\ref{thm:coloring} and~\ref{thm:testability of M2-free graphs} gives
  \begin{align*}
    m_{\ref{lem:test pinned}}(\epsilon,k)
    &=36(3k)^2\delta^{-2} = 36\cdot 3^4\cdot k^2y^4
      =36\cdot3^4\cdot k^2\cdot
      \left(4\frac{6 k}{\epsilon}\ln\frac{6k}{\epsilon}\right)^4\\
    &<
      2^{14}\cdot 3^{10}\cdot\epsilon^{-4}\cdot k^6\left(5\ln k\right)^4
    \\
    &< 2^{40}\cdot\epsilon^{-4}\cdot k^6\cdot\ln^4 k.\\
  \end{align*}
  We start by defining sections of~$\T$ by a set of~\emph{gates}.
  \begin{claim}\label{claim:gates}
    There is a set~$\G$ with~$\G_0\subseteq\G\subseteq\T$ of size at most~$3k$
    such that every connected component of~$\T\setminus\G$
    is an open segment~$(g_1,g_2)$ with~$g_1,g_2\in\G$.
  \end{claim}
  \begin{proof}
    Let~$\G_1$ be the set of leaves of~$\T$.
    We know by assumption that~$|\G_1|\leq k$.
    Since~$\T$ is a tree, there is a set~$\G_2\subseteq\T$
    of size at most~$|\G_1|-1$ such that
    every connected component of~$\T\setminus\G_1$
    is a topological segment.
    It suffices to set~$\G=\G_0\cup\G_1\cup\G_2$.
    Moreover, $|\G|\leq |\G_0|+|\G_1|+|\G_2|\leq 3k$.
  \end{proof}
  Let~$\T_1,\dots,\T_q$ be the connected components of~$\T\setminus\G$.
  Note that the section~$T_i$ can be seen as the edges of a tree with
  vertices~$\G$, so~$q=|\G|-1< 3k$.

  Let us construct a set coloring problem as defined in Section~\ref{sec:coloring}.
  To every vertex of~$u\in V$, we want to attribute
  a set of gates $\phi(u)\subseteq\G$ that corresponds to the set
  of gates that a subtree~$T_u$ in a representation~$(T_v)_{v\in V}$
  could.
  Informally speaking, $\phi(u)$ gives the general shape of~$T_u$.
  
  For every~$v\in V$, define
  \[
    L_v=\sst{\G\cap T}{\text{$T$ is a subtree of~$\T$ with~$x_v\in T$}}.
  \]
  Note that~$x_v\in L_v$ because~$x_v\in\G$ for every vertex~$v$.
  Let us define the associated constraints functions~$(m_{uv})$ and~$(M_{uv})$.
  Let $u,v\in V$ be distinct vertices.
  Let~$P_{uv}=\G\cap [x_u, x_v]$ be the set
  of gates on the path from~$x_u$ to~$x_v$ on~$\T$.

  Now define
  \begin{itemize}
  \item
    $m_{uv}(T) = P_{uv} \setminus T$
    and $M_{uv}(T) = \G$
    if~$uv\in E(G)$; and
  \item
    $m_{uv}(T) = \varnothing$ and
    $M_{uv}(T) = \G \setminus T$ if~$uv\notin E(G)$.
  \end{itemize}
  This definition is designed so that
  for a coloring~$\phi$, the relation
  \begin{equation}\label{eq:valid pair:2}
    m_{uv}(\phi(u))\subseteq \phi(v)\subseteq M_{uv}(\phi(u)).
  \end{equation}
  is satisfied
  if and only if~$P_{uv}\subseteq\phi(u)\cup\phi(v)$ when~$uv\in E(G)$;
  and if and only if~$\phi(u)\cap\phi(v)=\varnothing$
  when~$uv\in E(G)$.
  \begin{claim}\label{claim:pinned implies colorable}
    If~$G[X]$ has a $(x_v)_{v\in X}$-pinned representation,
    then $X$ has a proper coloring.
  \end{claim}
  \begin{proof}
    Let~$(T_v)_{v\in X}$ be a $(x_v)_{v\in X}$-pinned representation
    of~$G[X]$ on~$\T$.
    For every~$u\in X$, define
    \[
      \phi(u)=T_u\cap\G.
    \]
    The set~$T_u$ is a subtree of~$\T$ that contains~$x_u$, so
    it follows from the definitions that~$\phi(u)\in L_u$.
    To prove the claim, it remains to show
    that the coloring~$\phi$ is proper.
    Consider a pair $u,v\in X$ of distinct vertices and
    let us check that~\eqref{eq:valid pair:2} holds.

    If~$uv\in E(G)$, Relation~\eqref{eq:valid pair:2}
    is equivalent to~$P_{uv} \subseteq \phi(v) \cup \phi(u)$.
    Since~$T_u$ and~$T_v$ intersect, $x_u\in T_u$ and~$x_v\in T_v$,
    it holds that~$T_u\cup T_v \supseteq [x_u,x_v]$.
    Taking the intersection with~$\G$,
    it indeed gives~$\phi(u)\cup \phi(v) \supseteq P_{uv}$.

    If~$uv\notin E(G)$,
    it suffices to check that~$\phi(v)\cap\phi(u)=\varnothing$.
    It holds that~$T_u\cap T_v=\varnothing$ because $uv$ is not an edge of~$G$.
    It directly follows that~$\phi(u)$ and~$\phi(v)$ do not intersect
    as~$\phi(u)\subseteq T_u$ and~$\phi(v)\subseteq T_v$.
  \end{proof}
  \begin{claim}\label{claim:coloring disproof}
    If every coloring of~$G$
    has at least $\delta n^2$ conflicting pairs,
    then~$G[X]$ is not~$(x_v)_{v\in X}$-pinned
    with probability~$\frac{1}{2}$.
  \end{claim}
  \begin{proof}
    Recall that~$|\G|\leq 3k$ and~$|X|\geq m_{\ref{thm:coloring}}(\delta,3k)$.
    The claim then follows from Claim~\ref{claim:pinned implies colorable}
    and Theorem~\ref{thm:coloring}.
  \end{proof}
  By Claim~\ref{claim:coloring disproof},
  the theorem holds if
  every coloring of~$V$ has at least~$\delta n^2$ conflicting edges.
  In the following, we assume that there is a coloring~$\phi:V\to 2^{\G}$
  of~$V$ with at most~$\delta n^2$ conflicting edges.
  
  Given a section~$\T_i=(g_1,g_2)$, we consider the following sets.
  \begin{itemize}
  \item The set~$L_i$ of vertices~$u\in V$
    with~$g_1\in\phi(u)$ and~$g_2\notin\phi(u)$; and
  \item the set~$R_i$ of vertices~$u\in V$
    with~$g_1\notin\phi(u)$ and~$g_2\in\phi(u)$.
  \end{itemize}
  We then aim to apply Theorem~\ref{thm:testability of M2-free graphs}
  to the bipartite graph~$G[L_i,R_i]$.
  
  \begin{claim}\label{claim:test LMR}
    Let~$X\subseteq V$ and~$i\in[q]$.
    If the bipartite graph~$G[L_i\cap X, R_i\cap X]$ contains a~$M_2$
    and~$X$  contains no conflicting pair of~$\phi$,
    then~$G[X]$ contains an induced~$C_4$.
  \end{claim}
  \begin{proof}
    Assume that~$G$ has a~$M_2$ between in~$(L_i\cap X,R_i\cap X)$
    induced by the set~$\{\ell_1,\ell_2,r_1,r_2\}$
    with~$\ell_1,\ell_2\in L_i$ and~$r_1,r_2\in R_i$.
    Then by the definition of~$L_i$,
    both~$\phi(\ell_1)$ and~$\phi(\ell_2)$ contain~$g_1$.
    In particular~$\phi(\ell_1)\cap\phi(\ell_2)$ is non-empty.
    Since~$\ell_1\ell_2$ is not a conflicting pair,
    it follows from~\eqref{eq:valid pair:2}
    that~$\ell_1\ell_2$ is an edge of~$G$.
    A symmetric argument shows that~$r_1r_2$ is also an edge of~$G$.
    It follows that~$\{\ell_1,\ell_2,r_1,r_2\}$ induces a~$C_4$ in~$G$.
  \end{proof}

  \begin{claim}\label{claim:close to local interval}
    Let~$i\in[q]$.
    If one has to add/delete at least $\frac{\epsilon}{2q}n^2$ edges to
    the bipartite graph~$G[L_i,R_i]$ to make it $M_2$-free,
    then with probability at least~$\frac{1}{2}$,
    $G[X]$ is not chordal.
  \end{claim}
  \begin{proof}
    We actually show that~$G[X]$ contains an induced~$C_4$ with good probability
    using Claim~\ref{claim:test LMR}.
    
    Let~$Y$ be a set of vertices chosen uniformly at random among subsets of~$V$
    of size
    $y=m_{\ref{thm:testability of M2-free graphs}}(\frac{\epsilon}{2\cdot 3k})\geq m_{\ref{thm:testability of M2-free graphs}}(\frac{\epsilon}{2q})$.
    The probability that~$Y$ contains a conflicting edge of~$\phi$
    is bounded from above by the average number of conflicting edges in~$Y$,
    so
    \[
      \Pr(\text{$Y$ contains a conflicting edge of~$\phi$})
      \leq \frac{\binom{y}{2}}{\binom{n}{2}}\cdot\delta n^2<
      \delta y^2\leq\frac{1}{3}
    \]
    since~$\delta$ is appropriately
    defined as~$\delta=\frac{1}{3y^2}$.
    
    By Theorem~\ref{thm:testability of M2-free graphs},
    the bipartite graph~$G[L_i\cap Y, R_i\cap Y]$
    contains an induced~$M_2$ with probability at least~$\frac{1}{2}$.
    Applying Claim~\ref{claim:test LMR} and by the union bound,
    \[
      \Pr(\text{$G[Y]$ contains no $C_4$}) \leq \frac{1}{2}+\frac{1}{3}
      =\frac{5}{6}.
    \]
    It remains to amplify this result.
    Consider four independent random sets
    of vertices~$Y_1,\dots,Y_4$ in~$X$ each distributed as uniformly
    chosen among subsets of~$V$ of size~$y$.
    This is possible because~$4y\leq |X|$.
    The probability that for every~$i\in\{1,\dots,4\}$,
    the graph~$G[Y_i]$ has no induced~$C_4$
    is at most~$\left(\frac{5}{6}\right)^4<\frac{1}{2}$.
  \end{proof}
  By Claim~\ref{claim:close to local interval},
  we may assume that for every~$i$
  there is a $M_2$-free bipartite graph~$H_i$ on~$L_i\cup R_i$
  that differs from~$G[L_i,R_i]$ in at most~$\frac{\epsilon}{2q}n^2$
  pairs.
  By Theorem~\ref{thm:structure of M2-free graphs}\eqref{it:intervals},
  $H_i$ is the intersection graph of a family~$(I_u^i)_{u\in L_i\cup R_i}$
  of intervals of~$[0,1]$ with $0\in I_u^i$ for every~$u\in L_i$
  and~$1\in I_u^i$ for every~$u\in R_i$.
  Assuming that~$\T_i$ is the open segment~$(g_1,g_2)$ with $g_1,g_2\in\G$,
  we map these intervals to~$[g_1,g_2]$ by an
  homeomorphism~$f_i:[0,1]\to[g_1,g_2]$ such that~$f_i(0)=g_1$ and~$f(1)=g_2$
  This gives a
  family~$(f_i(I^i_u))_{u\in L_i\cup R_i}$ of paths of~$\T$
  of the form~$f_i(I^i_u)=[\ell_u^i,r_u^i]$ with~$\ell_u^i,r_u^i\in [g_1,g_2]$
  that satisfies the following properties
  \begin{itemize}
  \item $H_i$ is the intersection graph of the
    family~$([\ell_u^i,r_u^i])_{u\in L_i\cup R_i}$;
  \item $\ell_u^i=g_1$ whenever~$u\in L_i$; and
  \item $r_u^i=g_2$ whenever~$u\in R_i$.
  \end{itemize}
  We now construct a family of trees~$(T_u)_{u\in V}$ by gluing
  the previously defined elements as follows.
  For~$u\in V$,
  let~$T_{\phi(u)}$ be the minimal subtree of~$\T$ that contains~$\phi(u)$,
  then define
  \[
    T_u = T_{\phi(u)}\cup \left(\bigcup_i~[\ell_u^i,r_u^i]\right),
  \]
  where the union is taken over all indices~$i$
  such that~$u\in L_i\cup R_i$.
  
  Let us check that~$T_u$ is a subtree of~$\T$.
  If~$u\in L_i\cup R_i$, then the interval~$[\ell_u^i,r_u^i]$ contains an element of~$\phi(u)\subseteq\T_{\phi(u)}$ (this element is $\ell_u^i$ if~$u\in L_i$ and~$r_u^i$ if $u\in R_i$), so~$T_u$ is a connected part of~$\T$, and further a tree.
  
  Let~$F$ be the chordal graph that is the intersection graph of the family~$(T_u)_{u\in V}$. It remains to show that~$G$ is~$\epsilon$-close to~$F$.
  \begin{claim}\label{claim:error pairs}
    If a pair~$uv$ is in~$E(F)\triangle E(G)$,
    then one of the following situations occurs:
    \begin{enumerate}
    \item\label{it:error:phi}
      $uv$ is a conflicting pair of~$\phi$; or
    \item\label{it:error:Hi}
      $uv$ is on edges
      of~$E(H_i)\triangle E(G[L_i,R_i])$ for some~$i\in[q]$.
    \end{enumerate}
  \end{claim}
  \begin{proof}
    First note that the bipartite graph~$F[L_i,R_i]$
    is identical to~$H_i$ for every~$i\in[q]$.
    To see this, it suffices to check that for
    every~$u\in L_i$ and~$v \in R_i$, the trees~$T_u$ and~$T_v$
    intersects if and only
    if~$[\ell_u^i,r_u^i]$ and~$[\ell_u^i,r_u^i]$ intersects.
    It follows from the definitions
    that for every~$w\in L_i\cup R_i$
    the tree~$T_{\phi(w)}$ does not intersect~$\T_i$,
    so~$T_w\cap\T_i=[\ell_w^i,r_w^i]$.
    If~$u\in L_i$ and~$v\in R_i$, the trees~$T_{u}$ and~$T_{v}$
    can only intersect on~$\T_i$, so
    $T_u\cap T_v=[\ell_u^i,r_u^i]\cap[\ell_u^i,r_u^i]$.
    
    As a consequence, Situation~\ref{it:error:Hi} covers every case
    where~$(u,v)$ is a pairs of~$L_i\times R_i$ for some~$i\in[q]$. 
    
    If~$\phi(u)$ and~$\phi(v)$ intersect in~$g\in\G$,
    then by construction~$T_u$ and~$T_v$ intersects in the same element~$g$,
    so~$uv\in E(F)$.
    If~$uv\notin E(G)$ then
    the constraint~\eqref{eq:valid pair:2},
    that implies~$\phi(v)\cap\phi(u)=\varnothing$,
    is not satisfied.
    Further, $uv$ is a conflicting pair of~$\phi$.
    We now assume that~$\phi(u)\cap\phi(v)=\varnothing$.

    If an element~$g_1$ of~$\phi(u)$ is adjacent of an element~$g_2$
    of~$\phi(v)$ in the sense that~$(g_1,g_2)$ correspond to a section~$\T_i$,
    then (up to switching the roles played by~$u$ and~$v$)
    $u\in L_i$ and~$v\in R_i$.
    We now assume that~$\phi(u)$ has no adjacent element in~$\phi(v)$.
    
    Recall that~$P_{uv}=[x_u,x_v]\cap\G$, so in this case there is a gate
    $g$ is the path~$P_{uv}$ that does not belong to~$\phi(u)\cup\phi(v)$.
    Further, the gate~$g$
    separates the trees~$T_u$ and~$T_v$ in~$\T$,
    so~$uv\notin E(F)$.
    Moreover, $P_{uv}\nsubseteq \phi(u)\cup\phi(v)$, so
    it follows from the definition of~$m_{uv}$
    that $uv$ is a conflicting pair of~$\phi$ unless~$uv\in E(G)$.
  \end{proof}
  It remains to use Claim~\ref{claim:error pairs}
  to estimate the number of pairs in~$|E(F)\triangle E(G)|$.
  The number of pairs in Situation~\ref{it:error:phi}
  is at most~$\delta n^2$;
  the number of pairs in Situation~\ref{it:error:Hi}
  is at most~$\frac{\epsilon}{2q}n^2$ for each~$i\in[q]$
  and therefore at most~$\frac{\epsilon}{2}n^2$ in total.
  It follows that
  \begin{align*}
    |E(F)\triangle E(G)|
    \leq \delta n^2 + \frac{\epsilon}{2}n^2
    \leq \epsilon n^2.
  \end{align*}
\end{proof}

\section{Proof of Theorem~\ref{thm:main}}
\label{sec:main proof}

  \resetClaimCounter
  Assume that~$G$ is $\epsilon$-far from being a chordal graph and let
  us prove that when~$\epsilon$ is small enough,
  a set~$U$ chosen uniformly at random
  among subsets of~$V$ of size~$m:=\epsilon^{-25}$ induces a graph~$G[U]$
  that is not chordal.
  
  We plan to apply Lemma~\ref{lem:nearly simplicial}
  with parameter~$\delta$ such that $6\delta^{1/2}=\frac{\epsilon}{2}$.
  To this end, set~$\delta=\frac{\epsilon^2}{144}$.
  Define
  $X=\sst{u\in V}{p_G(u)< \frac{\delta}{2}\cdot n^2}$
  and $Y=V\setminus S$.
  We first prove that~$G[Y]$ is far from being an interval graph.
  \begin{claim}\label{claim:GY delta far}
    Adding/deleting $\frac{\delta}{2}n^2$ or less
    edges to~$G[Y]$ does not make it a chordal graph.
  \end{claim}
  \begin{proof}
    Assume for a contradiction that it is possible
    to add/delete at most~$\frac{\delta}{2}n^2$
    edges of~$G[Y]$ to~$G$ to obtain a graph~$G'$ on~$V$
    such that~$G'[Y]$ is chordal (so~$G$ and~$G'$ differs only inside the vertex set~$Y$).

    An induced copy of~$P_3$ with middle vertex~$u\in X$ in the graph~$G'$
    is either an induced~$P_3$ of~$G$ with middle vertex~$u$ or is the union of~$u$
    and a pair of~$E(G)\setminus E(G')$.
    As a consequence,
    \[
      p_{G'}(u)\leq p_G(u) + |E(G)\setminus E(G')|\leq
      \frac{\delta}{2}n^2 + \frac{\delta}{2}n^2\leq \delta n^2
    \]
    for every~$u\in X$.
    By Lemma~\ref{lem:nearly simplicial},
    the graph~$G'$ is therefore $\frac{\epsilon}{2}$-close to a chordal graph~$G''$.
    Since~$\frac{\epsilon}{2}+\frac{\delta}{2}<\epsilon$,
    it follows that~$G$ is $\epsilon$-close to~$G''$.
    As~$G''$ is chordal, this yields a contradiction.
  \end{proof}
  Let~$S$ be a subset of~$V$.
  Let~$Y_S$ be the set defined as in Lemma~\ref{lem:YS} by
  $Y_S=\sst{u\in V}{\text{$N(u)\cap S$ is not a clique of~$G$}}$.
  Let~$\R(S)$ be the set of minimal chordal representations of~$G[S]$, considered up to homeomorphism.
  Recall that an element~$R$ of~$\R(S)$ is a family~$R=(T_u)_{u\in S}$ of
  subtrees of a topological tree~$\T_R$.
  By Lemma~\ref{lem:chordal repr}, we know that
  that~$|\R(S)|\leq m_{\ref{lem:chordal repr}}(s)$.
  
  Fix a chordal representation~$R\in\R(S)$.
  By Lemma~\ref{lem:YS}
  there is a set~$\G_R\subseteq\T_R$ of size at most~$\binom{|S|}{2}$ and a
  family~$(x_{v}^R)_{v\in Y_S}$ of elements of~$\G_R$
  such that for every~$U\subseteq V$,
  if~$G[(U\cap Y_S)\cup S]$ has a chordal representation
  that extends~$R$ on a tree~$\T'$ that extends~$\T_R$ then
  the graph~$G[U\cap Y_S]$
  has a $(x_v^R)_{v\in U \cap Y_S}$-pinned representation on~$\T_R$.

  \begin{claim}\label{claim:GU to xR}
    Fix two sets of vertices~$S$ and~$U$.
    If~$G[S\cup U]$ is a chordal graph then there exists~$R\in\R(S)$
    such that~$G[U\cap Y_S]$
    has a~$(x_v^R)_{v\in U\cap Y_S}$-pinned representation on~$\T_R$.
  \end{claim}
  \begin{proof}
    Assume that~$G[S\cup U]$ has a chordal
    representation~$(T_u)_{u\in S\cup U}$ on a tree~$\T$.
    The subtree family~$(T_u)_{u\in S}$
    is in particular a chordal representation of~$G[S]$,
    so there is a subtree~$\T'\subseteq \T$ such
    that the family~$(T_u\cap\T')_{u\in S}$ of subtrees of~$\T'$ is a minimal
    representation of~$G[S]$.
    Up to applying an homeomorphism to~$\T$ and~$(T_u)_{u\in U}$, we may assume
   that~$R=(T_u\cap\T')_{u\in S}$ and~$\T'=\T_R$ for some~$R\in\R(S)$.
    In this case, the family~$(T_u)_{u\in (U\cap Y_S)\cup S}$ is a chordal representation
    of~$G[(U\cap Y_S)\cup S]$
    that extends~$R$ on the tree~$\T$ that extends~$\T_R$.
    It follows from the construction of~$(x_{v}^R)_{v\in Y_S}$ via Lemma~\ref{lem:YS}
    that~$G[U\cap Y_S]$ has a $(x_v^R)_{v\in U\cap Y_S}$-pinned representation
    on the tree~$\T_R$, as claimed.
  \end{proof}
  For~$S,U\subseteq V$ and a chordal representation~$R\in\R(S)$, define the event
  \[
    E_R(S,U): \text{$G[U\cap Y_S]$
      has a~$(x_u^R)_{u\in U\cap Y_S}$-pinned representation on~$\T_R$.}
  \]
  By Claim~\ref{claim:GU to xR}, $G[S\cup U]$ is chordal only if~$E_R(S,U)$ happens for some~$R\in\R(S)$.
  
  We aim to apply Lemma~\ref{lem:test pinned} to show that~$\Pr(E_R(S,U))$ is small when~$S$ and~$R\in\R(S)$ are fixed and~$U$ is a large enough random set.
  To have this lemma formally apply,
  we proceed to following padding construction.
  Let~$(y_v^R)_{v\in V}$ be a family of points of~$\G_R$  of size
  at most~$\binom{|S|}{2}$
  such that~$y_v^R=x_v^R$ for every~$v\in Y_S$ and that is arbitrarily defined
  on~$V\setminus Y_S$.
  Let~$H$ be the graph on~$V$ such that~$H[Y]=G[Y]$ and every vertex of~$X$
  is universal.

  Given~$S,U\subseteq V$ and~$R\in\R(S)$, define the event
  \[
    \text{$E_R'(S,U)$:~$G[U\cap Y]$ has
      a~$(y_v^R)_{v\in U\cap Y}$-pinned representation on~$\T_R$.}
  \]
  \begin{claim}\label{claim:y-pinned:U0}
    Fix~$S\subseteq V$ of size~$s$ and~$R\in\R(S)$.
    Let~$U_0$ be a random set chosen uniformly among subsets
    of~$V$ of size~$m_{\ref{lem:test pinned}}(\frac{\delta}{2},\frac{s^2}{2})$.
    The probability of~$E'_R(S,U_0)$ is at most~$\frac{1}{2}$.
  \end{claim}
  \begin{proof}
    Observe first that~$H$ is $\frac{\delta}{2}$-far from being a chordal graph
    because by Claim~\ref{claim:GY delta far}
    one has to add/delete at least~$\frac{\delta}{2}n$
    edges of~$H[Y]=G[Y]$ to make it chordal.
    Moreover, $|\G_R|\leq\binom{s}{2}$ and by Proposition~\ref{prop:leaf} the tree~$\T_R$ has at most~$s$ leaves.
    By Lemma~\ref{lem:test pinned} applied with~$k=s^2/2$, this implies
    that with probability at least~$\frac{1}{2}$,
    the graph~$H[U_0]$ has no
    $(y_v^R)_{v\in U_0}$-pinned representation on~$\T_R$.
    
    Now observe that a $(y_v^R)_{v\in U_0\cap Y}$-pinned
    representation~$(T_v)_{v\in U_0\cap Y}$ of
    the graph~$G[U_0\cap Y]=H[U_0\cap Y]$ on~$\T_R$
    can be extended to a $(y_v^R)_{v\in U_0}$-pinned representation
    of~$H[U_0]$ on~$\T_R$ by defining~$T_v:=\T_R$ for every~$v\in X\cap U_0$.
    
    The claim follows directly from these two observations.
  \end{proof}
  Let us amplify Claim~\ref{claim:y-pinned:U0}.
  \begin{claim}\label{claim:prob ER'}
    Fix~$S\subseteq V$ of size~$s$ and~$R\in\R(S)$.  
    Let~$U$ be a set chosen uniformly at random among subsets of~$V$ of
    size~$t=(2+\log_2(m_{\ref{lem:chordal repr}}(s)))\cdot m_{\ref{lem:test pinned}}(\frac{\delta}{2},\frac{s^2}{2})$,
    then
    \[
      \Pr(E_R'(S,U))\leq \frac{1}{4m_{\ref{lem:chordal repr}}(s)}.
    \]
  \end{claim}
  \begin{proof}
    We consider that~$U$ contains $\ell:=2+\log_2(m_{\ref{lem:chordal repr}}(s))$
    independent subsets $U_1,\dots,U_{\ell}$
    of size~$m_{\ref{lem:test pinned}}(\frac{\delta}{2},\frac{s^2}{2})$.
    The events~$(E_R'(S,U_i))_{i=1}^\ell$ are therefore mutually independent.
    As a consequence of Claim~\ref{claim:y-pinned:U0},
    the probability of~$E_R'(S,U_i)$ is at most~$1/2$ for every~$i\in[\ell]$.
    It then follows that
    \[
      \Pr(E_R'(S,U))\leq\Pr(\bigcap_{i=1}^{\ell}E_R'(S,U_i))
      =\prod_{i=1}^{\ell}\Pr(E_R'(S,U_i))\leq 2^{-\ell}=
      \frac{1}{4m_{\ref{lem:chordal repr}}(s)}.
    \]
  \end{proof}
  The events~$E_R(S,U)$ and~$E_R'(S,U)$ are related by the following property.
  \begin{claim}\label{claim:ER and ER'}
    For every $S,U\subseteq V$ and~$R\in\R(S)$,
    the event~$E_R(S,U)$ implies~$E_{R}'(S,U)$ whenever~$U\cap Y\subset U\cap Y_S$.
  \end{claim}
  \begin{proof}
    Indeed, assume that~$U\cap Y\subset U\cap Y_S$ and that~$E_R(S,U)$ holds, that is
    $G[U\cap Y_S]$ has a~$(x_u^R)_{u\in U\cap Y_S}$-pinned
    representation~$(T_u)_{u\in U\cap Y_S}$ on~$\T_R$.
    The restriction~$(T_u)_{u\in U\cap Y}$ of this representation to~$G[U\cap Y]$
    is a~$(x_v^R)_{v\in U\cap Y}$-pinned representation of~$G[U\cap Y]$ on the same tree~$\T_R$.
    Since~$x_v^R=y_v^R$ for every~$v\in U\cap Y \subseteq Y_S$, the same family~$(T_u)_{u\in U\cap Y}$ is also a~$(y_v^R)_{v\in U\cap Y}$-pinned representation of~$G[U\cap Y]$ on~$\T_R$, so~$E_R'(S,U)$ holds.
  \end{proof}
  Let~$E_Y(S,U)$ be the event that $U\cap Y\subset U\cap Y_S$.
  It follows from Claim~\ref{claim:ER and ER'} that
  for every~$S,U\subseteq V$ and~$R\in\R(S)$,
  \begin{equation}\label{eq:ER and ER'}
    \bigcup_{R\in\R(S)}E_R(S,U) \subseteq
    \left(\bigcup_{R\in\R(S)}E_R'(S,U)\right)\cap E_Y(S,U).
  \end{equation}
  Let us bound from below the probability of~$E_Y(S,U)$.
  \begin{claim}\label{claim:prob EY}
    Fix~$U\subseteq V$ and let~$S$ be a set chosen uniformly at random among the subsets of~$V$ of size~$s\leq\frac{8}{\delta}\ln|U|$.
    Then
    $Pr(E_Y(S,U)) \geq 1 - \frac{1}{4}$.
  \end{claim}
  \begin{proof}
    Note that a vertex~$u\in Y\cap U$ belongs to~$Y_S$ if and only if
    there is a~$P_3$ in~$G[S]$ on vertices~$\{s_1,u,s_2\}$
    whose middle vertex is~$u$ and such that~$s_1$ and~$s_2$ are in~$S$.
    By definition of~$Y$, it holds that~$p_G(u)\geq\frac{\delta}{2}n^2$ in the graph~$G$.
    The random set~$S\setminus\{u\}$
    contains at least~$\frac{s-2}{2}$ independent pairs of vertices,
    i.e. the probability that~$u\notin Y_S$, i.e. that none of these pairs forms a~$P_3$ with middle vertex~$u$, is at most
    \[
      \left(1-\frac{\delta}{2}\cdot n^2/\binom{n}{2}\right)^{\frac{s-2}{2}}
      \leq (1-\delta)^{\frac{s}{4}}
      \leq e^{-\frac{\delta}{4}s}\leq |U|^{-2}< \frac{1}{4|U|}.
    \]
    By the union bound, it follows that
    \[
      \Pr(E_Y(S,U)) \geq 1 - |U|\cdot\frac{1}{4|U|} = \frac{3}{4}.
    \]
  \end{proof}
  We are now ready to prove the theorem.
  Recall that~$r=\epsilon^{-100}$.
  Set $s=\frac{8}{\delta}\ln r$ as suggested by Claim~\ref{claim:prob EY}
  and $t=(2+\log_2(m_{\ref{lem:chordal repr}}(s)))\cdot m_{\ref{lem:test pinned}}(\frac{\delta}{2},\frac{s^2}{2})$ as in Claim~\ref{claim:prob ER'}.

  We claim that
  \begin{equation}\label{eq:s, t and r}
    s+t\leq r
  \end{equation}
  when~$\epsilon$ is small enough.
  Before proving~\eqref{eq:s, t and r}, let us see why it implies the theorem.
  In this case, $W$ contains two sets~$S$ and~$U$ chosen independently and uniformly at random among subsets of respective size~$s$ and~$t$.

  Claim~\ref{claim:prob EY} applies to~$S$ and~$U$ because~\eqref{eq:s, t and r}
  implies that~$t\leq r$, and further that~$s\geq\frac{8}{\delta}\ln t$.
  Consequently, it holds that~$\Pr(E_Y(S,U))\leq\frac{3}{4}$.
  By Claim~\ref{claim:prob ER'} and the union bound,
  the probability that~$E_R'(S,U)$ holds for at least one~$R\in\R(S)$ is at most $|\R(S)|\cdot\frac{1}{4m_{\ref{lem:chordal repr}}(s)}\leq\frac{1}{4}$.
  Using~\eqref{eq:ER and ER'} and Claim~\ref{claim:GU to xR}, the probability that~$G[S\cup U]$ is a chordal graph is at most
  \[
    \Pr(\exists R\in\R(S)\text{ s.t. }E'(S,U))
    + (1 - \Pr(E_Y(U,S))) \leq \frac{1}{4}+\frac{1}{4}=\frac{1}{2}.
  \]
  To conclude the proof of Theorem~\ref{thm:main},
  it remains to show~\eqref{eq:s, t and r}.
  We have
  \begin{align*}
    t+s
    &= (2+\log_2(m_{\ref{lem:chordal repr}}(s)))\cdot
      m_{\ref{lem:test pinned}}\left(\frac{\delta}{2},\frac{s^2}{2}\right) + s\\
    &\leq (2+2s^2\ln(3s))\cdot 2^{40}\cdot\left(\frac{\delta}{2}\right)^{-4}
      \cdot\left(\frac{s^2}{2}\right)^6\ln^{4}\left(\frac{s^2}{2}\right) + s\\
    &\leq O\left(s^{14}\ln^{5}s\cdot\delta^{-4}\right)\\
    &\leq O\left(\delta^{-18}\ln^{14}(1/\epsilon)\cdot\ln^5(\ln(\epsilon)/\delta)\right)\\
    &\leq O\left(\epsilon^{-36}\ln^{19}(1/\epsilon)\right)
  \end{align*}
  which is indeed smaller than~$\epsilon^{-37}$ when~$\epsilon$ is small enough.

\qed

\bibliographystyle{plain}
\bibliography{bibli}
\appendix

\section{Proof of Theorem~\ref{thm:coloring}}
\label{ap:proof:thm:coloring}

In this section we prove Theorem~\ref{thm:coloring}.
This proof is an adaptation
of the proof in~\cite[Section~4]{AlonK02}.
We use the same notations to make the similarity apparent.
\begin{proof}
  \resetClaimCounter
  Set~$p=\max_{u\in V}|L_u|$.
  Assume that every coloring of~$V$ has at least~$\epsilon n^2$ conflicting pairs.
  The goal of the proof is to show that if~$X$ is a random subset of~$V$
  of size~$s=36k\ln p/\epsilon^2$,
  then $G[X]$ has no proper coloring with probability
  at least~$\frac{1}{2}$.
  For the sake of the analysis,
  we consider that~$X$ is generated
  in~$s$ rounds, each time choosing a new vertex~$x_j$ in~$V$,
  uniformly at random.

  Note that we may assume that~$\epsilon<\frac{1}{2}$
  and~$p\geq2$.
  Indeed, if~$\epsilon\geq\frac{1}{2}$ then any coloring of~$V$
  has at most~$\binom{n}{2}<\epsilon n^2$ conflicting pairs.
  If~$p=1$, then~$V$ has a unique coloring~$\phi$,
  and it suffices to show that~$X$ contains one of
  the (at least)~$\epsilon n^2$ conflicting
  pairs of~$\phi$ with probability~$\frac{1}{2}$.
  It suffices for this that~$|X|\geq\epsilon^{-1}$.
  Indeed, in this case~$X$ contains at least~$\frac{1}{2\epsilon}$
  independent pairs of vertices, so the probability that none of them is conflicting
  is at most~$(1-2\epsilon)^{1/(2\epsilon)}\leq e^{-1}<\frac{1}{2}$.

  Given a proper coloring~$\phi:S\to 2^{[k]}$
  of a set~$S\subseteq V$,
  we analyze how~$\phi$ can be extended to a larger subset of~$V$.
  For a vertex~$v\in V$,
  let $L_\phi(v) \subseteq L_v$ be the list of colors~$c$ such that extending~$\phi$
  to~$S\cup\{v\}$ by~$\phi(v)=c$
  is a proper coloring of~$S\cup\{v\}$.
  More precisely, this set is defined by
  \[
    L_\phi(v) =
    \sst{c \in L_v}{\forall u\in S,~\text{$c$ and $\phi(u)$ are not conflicting}}.
  \]
  Now define the sets
  \[
    m_\phi(v)=\bigcup_{u\in S}m_{uv}(\phi(u))  \quad\text{and}\quad
    M_\phi(v)=\bigcap_{u\in S}M_{uv}(\phi(u)).
  \]
  Note that it follows from the definitions and~\eqref{eq:valid pair}
  that every color~$c\in L_\phi(v)$
  satisfies~$m_\phi(v)\subseteq c \subseteq M_\phi(v)$.
  We will therefore use the value
  \[
    E_\phi:=\sum_{v\in V}\left|M_\phi(v)| - |m_\phi(v)\right|
  \]
  as a measure of the freedom we have
  when trying to extend the partial coloring~$\phi$ to~$V$.
  
  A vertex~$v$ with $L_\phi(v)=\varnothing$ is called \emph{colorless}.
  If~$v$ is a colorless vertex, then we know that~$\phi$
  does not extend properly to~$S\cup\{v\}$.
  Let~$U_\phi$ be the set of colorless vertices.

  Given~$u\in V$ and a color~$c\in L_u$, define
  \[
    \delta^c_\phi(u) = \sum_{v \in V}
    |M_\phi(v)\setminus M_{uv}(c)| + |m_{uv}(c)\setminus m_\phi(v)|.
  \]
  This value is defined so that if~$\phi$ is extended to a coloring~$\phi'$
  on~$S\cup\{u\}$ by~$\phi'(u)=c$,
  then $E_{\phi'}=E_\phi - \delta^c_\phi(u)$.
  
  A greedy coloring of the whole vertex set~$V$
  extending~$\phi$ consists in assigning
  to a vertex~$v\in V\setminus (S\cup U_\phi)$ the color~$c=\alpha_\phi(v)$
  that minimizes the number~$\delta^c_\phi(v)$
  among the colors of~$L_\phi(v)$.
  Let~$\delta_\phi(v):=\min_{c\in L_\phi(v)}\delta^c_\phi(v)$ be this minimum.
  If~$u$ is colorless, then choose $\alpha_\phi(u)$ arbitrarily in~$L_u$.
  \begin{claim}\label{claim:greedy coloring}
    The number
    \[
      \sum_{u\in V\setminus(U_\phi\cup S)}\delta_\phi(u) + n|U_\phi|
    \]
    is an upper bound to the number of conflicting pairs of
    the coloring~$\alpha_\phi$ on~$V$.
  \end{claim}
  \begin{proof}
    The number of conflicting edges of~$\alpha_\phi$
    containing a vertex of~$U_\phi$ is at most~$n|U_\phi|$.
    Moreover, it follows from the definition
    of~$L_\phi$ and~$\alpha_\phi$
    that there is no conflict between~$S$ and~$V\setminus U_\phi$.
    
    Now,
    assume that a pair~$uv$ with $u,v\in V\setminus(U_\phi\cup S)$ is conflicting
    because the constraint
    $m_{uv}(\alpha_\phi(u))\subseteq\alpha_\phi(v)\subseteq M_{uv}(\alpha_\phi(u))$
    does not hold, i.e.
    \begin{equation}\label{eq:conflict uv}
      m_{uv}(\alpha_\phi(u))\nsubseteq\alpha_\phi(v)
      \text{\quad or \quad}
      \alpha_\phi(v)\nsubseteq M_{uv}(\alpha_\phi(u)).
    \end{equation}
    Recall that by definition of~$L_\phi(v)$, it holds that
    \begin{equation}\label{eq:alpha in L}
      m_\phi(v)\subseteq\alpha_\phi(v)\subseteq M_\phi(v).
    \end{equation}
    It follows from~\eqref{eq:conflict uv} and~\eqref{eq:alpha in L}
    that either $m_{uv}(\alpha_\phi(u))\nsubseteq m_\phi(v)$
    or $M_\phi(v)\nsubseteq M_{uv}(\alpha_\phi(u))$.
    In both cases,
    \[
      |M_\phi(v)\setminus M_{uv}(\alpha_\phi(u))| +
      |m_{uv}(\alpha_\phi(u))\setminus m_\phi(v)| \geq 1,
    \]
    so~$v$ contributes for at least one
    in the sum~$\delta_\phi^{\alpha_\phi(u)}(u)=\delta_\phi(u)$.
    It follows that the number of conflicting edges of~$\alpha_\phi$
    inside~$V\setminus(U_\phi \cup S)$
    is at most $\sum_{v\in V\setminus(U_\phi\cup S)}\delta_\phi(u)$.
  \end{proof}
  Let~$W_\phi$ be the set of vertices~$v\in V\setminus(U_\phi\cup S)$
  satisfying~$\delta_\phi(v)\geq\epsilon n/2$.
  A vertex in~$W_\phi$ is called \emph{restricting}.
  We deduce the following property.
  \begin{claim}\label{claim:U+W}
    For every proper partial coloring~$\phi$,
    \[
      |W_\phi|+|U_\phi|\geq \frac{\epsilon}{2}n.
    \]
  \end{claim}
  \begin{proof}
    Recall that every coloring of~$V$ is assumed to have at least~$\epsilon n^2$
    conflicting edges.
    Applying this hypothesis on~$\alpha_\phi$,
    it follows from Claim~\ref{claim:greedy coloring} that
    \begin{align*}
      \epsilon n^2
      & \leq \sum_{u\in V\setminus(U\cup S)}\delta_\phi(v) + n|U_\phi|\\
      & \leq \frac{\epsilon}{2}n^2 + n(|W_\phi|+|U_\phi|).
    \end{align*}
    It further holds that~$|W_\phi|+|U_\phi|\geq \frac{\epsilon}{2}n$.
  \end{proof}  
  We follow the argument of~\cite{AlonK02}.
  In the process of choosing~$r_1,\dots,r_s$,
  we construct an auxiliary tree~$T$ whose nodes
  have at most~$p$ children.
  Each node of~$T$ can be unlabeled or labeled
  with either a vertex of~$G$ or a special symbol~$\#$
  called the \emph{terminal symbol}.
  During all the process we maintain the following property:
  every leaf of~$\T$ is either unlabeled or labeled by the symbol~$\#$,
  and every inner node is labeled by a vertex of~$G$.
  A leaf labeled with~$\#$ is called a \emph{terminal node}
  and remains a terminal node during all the process while
  unlabeled leaves can be later labeled.
  Every edge of~$T$ is labeled by a subset of~$[k]$
  (i.e. a color), such that the label of
  an edge from a node labeled by~$v$ to one of its children
  is an element of~$L_v$.

  For a node~$t$ of~$T$, we define the set~$S_t\subseteq V$
  as the set of the labels of the nodes along the path from the root to~$t$.
  The labels of the edges on this path define a coloring of~$S_t$
  in the natural way:
  $\phi_t(v)$ is the label of the edge following the node labeled by~$v\in S_t$
  in the path to~$t$.
  Now, the coloring~$\phi_t$ defines a set~$U_{\phi_t}$ of colorless vertices and
  a set~$W_{\phi_t}$ of restricting vertices.
  
  Let us describe the construction of~$T$.
  The construction starts with a single node that is unlabeled.
  Round~$j$ is \emph{successful} for the unlabeled node~$t$
  if the sampled vertex~$r_j$ belongs to~$U_{\phi_t}\cup W_{\phi_t}$.
  If round~$j$ is successful for~$t$,
  then we label~$t$ by~$r_j$ and
  we add~$|L_v|$ children~$t_1,\dots,t_{|S_v|}$ of~$t$ where
  each edge~$tt_i$ is labeled by a different element of~$L_v$.
  For each of these newly-defined leaves~$t_i$,
  if the corresponding coloring~$\phi_{t_i}$
  is not proper then we label~$t_i$ by the terminal symbol~$\#$.
  Otherwise, $t_i$ is kept unlabeled.

  \begin{claim}\label{claim:E decreases}
    Let~$t$ be a node of~$T$ and~$t'$ a child of~$T$.
    Assume that~$t'$ is not a terminal node, then
    \[
      E_{\phi_{t'}}\leq E_{\phi_t} - \frac{\epsilon n}{2}.
    \]
  \end{claim}
  \begin{proof}
    First note that if the label~$u$ of~$t$ belongs to~$U_{\phi_t}$,
    then~$\phi_t$ does not extend to a proper coloring
    of~$S_{t'}=S_t\cup\{v\}$ and therefore~$t'$ is a terminal node.
    It follows that~$u$ belongs to~$W_{\phi_t}$. 
    As a consequence, $\delta_{\phi_t}(u) \geq \epsilon n/2$.
    Further,
    $E_{\phi_{t'}}\leq E_{\phi_t} -\delta_{\phi_t}(u)\leq E_{\phi_t} - \epsilon n/2$.
  \end{proof}  
  The \emph{depth} of a tree is the number of edges
  in a shortest path from the root to a leaf.
  \begin{claim}
    The depth of~$T$ is smaller than~$\frac{3k}{\epsilon}$.
  \end{claim}
  \begin{proof}
    Let~$t_0,\dots,t_d=t$ be the path with~$d$ edges from the root~$t_0$
    to a leaf~$t_d$ of~$T$.
    For every $i\in\{1,\dots,d-1\}$, the node~$t_i$ is not terminal,
    so we know by Claim~\ref{claim:E decreases},
    that~$E_{\phi_{t_{i}}}\leq E_{\phi_{t_{i-1}}} - \frac{\epsilon n}{2}$.
    Consequently,
    $E_{\phi_{t_{d-1}}} - E_{\phi_{t_0}} \geq  (d-1)\frac{\epsilon n}{2}$.
    Using that~$0\leq E_\phi\leq kn$ for every proper coloring~$\phi$,
    we obtain $kn \geq (d-1)\frac{\epsilon n}{2}$.
    It follows that~$d \leq \frac{2k}{\epsilon}+1<\frac{3k}{\epsilon}$.
  \end{proof}
  \begin{claim}\label{claim:terminal equals no coloring}
    If at the end of the process every leaf of~$T$ is a terminal node,
    then there is no proper coloring of~$X$.
  \end{claim}
  \begin{proof}
    Let~$c:S\to 2^{[k]}$ be a coloring of~$X$ and let us show that~$c$
    is not proper.
    We start at the root~$t_0$ of the tree~$T$ and
    we follow a path~$t_0,\dots,t_d=t$ to a leaf~$t$ satisfying the following rule.
    If~$t_i$ labeled by~$v_i\in X$,
    choose the child~$t_{i+1}$ of~$t_i$ such that the edge~$t_it_{i+1}$
    is labeled by~$c(v)$.
    It follows from the definition that~$\phi_{t}$ is
    equal to the restriction of the coloring~$c$
    to~$S_{t}=\{v_0,\dots,d_{d-1}\}$.
    By assumption, the leaf~$t$ is terminal, so~$\phi_{t}$ is not
    a proper coloring of~$S_{t}$,
    which implies that~$c$ is not a proper coloring of~$X$.
  \end{proof}
  \begin{claim}\label{claim:probability for terminal}
    With probability at least~$\frac{1}{2}$, all leaves are terminal nodes at
    the end of the process.
  \end{claim}
  \begin{proof}
    Since every node of~$T$ has at most~$p$ sons
    and the tree~$T$ has depth at most~$\frac{3k}{\epsilon}$,
    the tree~$T$ can be embedded in the $p$-ary tree~$T_{p,\frac{3k}{\epsilon}}$
    of depth~$\frac{3k}{\epsilon}$.
    Moreover, this embedding can be done before actually constructing~$T$.
    The number of leaves of this tree is~$p^{\frac{3k}{\epsilon}}$.

    Fix a leaf~$t$ of~$T_{p,\frac{3k}{\epsilon}}$
    and consider the path~$P_t$ from the root to~$t$.
    During the process, the depth of the leaf of~$T$ on this path
    is equal to the number of successful rounds for vertices of the path.
    If at round~$j$ there is no terminal node of~$T$ on this path,
    then the path contains a non-terminal leaf of~$T$
    and by Claim~\ref{prop:nonempty tree intersection}
    this leaf has probability of success at least~$\epsilon/2$.
    It follows that the probability that there is no terminal node
    on~$P_t$ at the end of the process is at most the probability
    that a binomial law $Y=B(r, \frac{\epsilon}{2})$ is less than~$3k/\epsilon$.
    By the Chernoff inequality, this is at most
    \[
      \Pr\left( Y < \frac{3k}{\epsilon} \right)
      \leq \exp\left( -\frac{1}{r\epsilon}
        \left(r\frac{\epsilon}{2} - \frac{3k}{\epsilon}\right)^2 \right).
    \]
    Injecting~$r=36\ln p \cdot k\epsilon^{-2}$, this upper bound becomes
    \begin{align*}
      \exp\left(
        -\frac{\left(18\ln p\cdot k\epsilon^{-1}- 3k\epsilon^{-1}\right)^2}{36\ln p\cdot k\epsilon^{-1}}
      \right) &= 
      \exp\left(
        -\frac{k \left(18\ln p - 3\right)^2}{36\ln p\cdot \epsilon}
                \right)\\
              &<
                \exp\left(
        -\frac{k (12\ln p)^2}{36\ln p\cdot \epsilon}
      \right)
                = p^{-\frac{4k}{\epsilon}}
    \end{align*}
    because~$p\geq 2$. Now, we deduce by the union bound that
    the probability that there is a non-terminal leaf in~$T$ at the end of
    the process it at most
    \[
      p^{-\frac{4k}{\epsilon}}\cdot p^{\frac{3k}{\epsilon}}
      = p^{-\frac{k}{\epsilon}}< p^{-1} \leq \frac{1}{2}.
    \]
    For this estimation, we used that~$\epsilon\leq\frac{1}{2}$ and~$p\geq 2$.
  \end{proof}
  The theorem then follows directly from
  Claims~\ref{claim:probability for terminal}
  and~\ref{claim:terminal equals no coloring}.
\end{proof}

\end{document}